\nonstopmode
\documentclass[10pt]{amsart}
\usepackage{graphicx}
\usepackage{latexsym}
\usepackage{fancyhdr}
\usepackage{amsmath, amssymb,amsthm}
\usepackage[all]{xy}
\usepackage{pdflscape}
\usepackage{longtable}
\usepackage{rotating}
\usepackage{verbatim}
\usepackage{hyperref}
\usepackage{cleveref}
\usepackage{subfigure}
\usepackage{mathrsfs}
\usepackage{mdwlist}
\usepackage{dsfont}
\usepackage{mathtools}
\usepackage{float}
\usepackage{color}
\usepackage{stmaryrd}

\definecolor{teal}{rgb}{0.0, 0.5, 0.5}

\newcounter{mnotecount}[section]

\newcommand{\rmnote}[1]{}

\allowdisplaybreaks

\DeclareFontFamily{U}{mathb}{\hyphenchar\font45}
\DeclareFontShape{U}{mathb}{m}{n}{
      <5> <6> <7> <8> <9> <10> gen * mathb
      <10.95> mathb10 <12> <14.4> <17.28> <20.74> <24.88> mathb12
      }{}
\DeclareSymbolFont{mathb}{U}{mathb}{m}{n}
\DeclareFontSubstitution{U}{mathb}{m}{n}

\let\dot\relax
\DeclareMathAccent{\dot}{0}{mathb}{"39}
\let\ddot\relax
\DeclareMathAccent{\ddot}{0}{mathb}{"3A}
\let\dddot\relax
\DeclareMathAccent{\dddot}{0}{mathb}{"3B}
\let\ddddot\relax
\DeclareMathAccent{\ddddot}{0}{mathb}{"3C}

\theoremstyle{plain}
\newtheorem*{theorem*}{Theorem}
\newtheorem{theorem}{Theorem}[section]
\newtheorem*{lemma*}{Lemma}
\newtheorem{lemma}[theorem]{Lemma}
\newtheorem*{assumption*}{Assumption}

\newtheorem*{proposition*}{Proposition}
\newtheorem{proposition}[theorem]{Proposition}
\newtheorem*{corollary*}{Corollary}
\newtheorem{corollary}[theorem]{Corollary}
\newtheorem*{claim*}{Claim}

\newtheorem*{conjecture*}{Conjecture}

\newtheorem*{question*}{Question}
\newtheorem*{result*}{Result}

\theoremstyle{definition}
\newtheorem*{definition*}{Definition}

\newtheorem*{example*}{Example}

\newtheorem*{algorithm*}{Algorithm}
\newtheorem*{remark*}{Remark}
\newtheorem*{remarks*}{Remarks}
\newtheorem{remark}[theorem]{Remark}

\newtheorem*{convention*}{Convention}

\numberwithin{equation}{section}

\def\al{\alpha}
\def\be{\beta}
\def\ga{\gamma}
\def\de{\delta}

\def\ve{\varepsilon}

\def\ka{\kappa}
\def\la{\lambda}

\def\rh{\rho}

\def\si{\sigma}

\def\ta{\tau}

\def\vh{\varphi}

\def\om{\omega}

\def\La{\Lambda}

\def\Ph{\Phi}

\def\C{\mathbb{C}}

\def\N{\mathbb{N}}

\def\R{\mathbb{R}}

\def\cA{\mathcal{A}}

\def\cD{\mathcal{D}}
\def\cE{\mathcal{E}}
\def\cF{\mathcal{F}}

\def\cH{\mathcal{H}}

\def\cS{\mathcal{S}}

\def\fB{\mathfrak{B}}

\def\fH{\mathfrak{H}}

\def\fK{\mathfrak{K}}
\def\fL{\mathfrak{L}}
\def\fM{\mathfrak{M}}
\def\fN{\mathfrak{N}}

\def\fQ{\mathfrak{Q}}

\def\fS{\mathfrak{S}}

\def\<{\langle}
\def\>{\rangle}
\renewcommand{\o}{\circ}

\def\ul{\underline}

\let\on=\operatorname

\newcommand{\sr}[1]%
{\ifmmode{}^\dagger\else${}^\dagger$\fi\ifvmode
\vbox to 0pt{\vss
 \hbox to 0pt{\hskip\hsize\hskip1em
 \vbox{\hsize3cm\raggedright\pretolerance10000
 \noindent #1\hfill}\hss}\vss}\else
 \vadjust{\vbox to0pt{\vss%
 \hbox to 0pt{\hskip\hsize\hskip1em%
 \vbox{\hsize3cm\raggedright\pretolerance10000%
 \noindent #1\hfill}\hss}\vss}}\fi%
}

\def\RR{\mathbb R}
\def\NN{\mathbb N}

\def\A{\;\forall}
\def\E{\;\exists}

\providecommand{\mapsfrom}{\kern.2em%
\setbox0=\hbox{$\leftarrow$\kern-.10em\rule[0.26mm]{0.1mm}{1.3mm}}\box0%
\kern.3em}

\title[Optimal solutions of the Borel problem]
{On optimal solutions of the Borel problem\\ in the Roumieu case}

\author[D.N.~Nenning, A.~Rainer, and G.~Schindl]{David Nicolas Nenning, Armin Rainer, and Gerhard Schindl}

\address{Fakult\"at f\"ur Mathematik, Universit\"at Wien, Oskar-Morgenstern-Platz~1, A-1090 Wien, Austria.}
\email{david.nicolas.nenning@univie.ac.at}
\email{armin.rainer@univie.ac.at}
\email{gerhard.schindl@univie.ac.at}

\begin{document}

\begin{abstract}
  The Borel problem for Denjoy--Carleman and Braun--Meise--Taylor classes has well-known optimal solutions.
  The unified treatment of these ultradifferentiable classes by means of one-parameter families of weight sequences
  allows to compare these optimal solutions.
  We determine the relations among them and give conditions for their equivalence in the Roumieu case.
\end{abstract}

\thanks{AR was supported by FWF-Project P 32905-N, DNN and GS by FWF-Project P 33417-N}
\date{\today}

\maketitle

\section{Introduction}\label{Introduction}

The Borel map $j^\infty : C^\infty(\R) \to \C^\N$ takes a smooth function $f$ to its infinite jet $(f^{(n)}(0))_{n \in \N}$ at $0$.
We will be concerned with the restriction of $j^\infty$ to Denjoy--Carleman classes $\cE^{\{M\}}(\R)$,
Braun--Meise--Taylor classes $\cE^{\{\om\}}(\R)$, and, most generally, classes $\cE^{\{\fM\}}(\R)$,
where $\fM$ is a one-parameter family of weight sequences.
We will only treat the Roumieu case in this paper; the Beurling case will be discussed in a separate paper.
The $j^\infty$-image of any of these ultradifferentiable classes sits in a sequence space
$\La^{\{M\}}$, $\La^{\{\om\}}$, and $\La^{\{\fM\}}$
defined by analogous bounds.
The (mixed) Borel problem asks for conditions for the validity of the inclusion
\begin{equation} \label{eq:problem}
      \La^{\{\text{weight}'\}} \subseteq j^\infty \cE^{\{\text{weight}\}}(\R)
\end{equation}
in any of the above cases, where different weights may appear on the left and on the right.

For Denjoy--Carleman and Braun--Meise--Taylor classes the \emph{optimal} solution to this problem is well-known:
\begin{enumerate}
  \item Let $M$ be a non-quasianalytic weight sequence.
  If $M'$ is a suitable other weight sequence, then $\La^{\{M'\}} \subseteq j^\infty \cE^{\{M\}}(\R)$ is equivalent to
  a condition (namely \eqref{eq:SV}) purely in terms of $M',M$ which we denote by $M' \prec_{SV} M$; see \cite{surjectivity}.
  There is an explicit positive sequence $L=L(M)$ such that $L \prec_{SV} M$, that is
  $\La^{\{L\}} \subseteq j^\infty \cE^{\{M\}}(\R)$, and $L$ is optimal with respect to this property;
  see \cite{maximal}.
  \item The condition $M' \prec_{\ga_1} M$ (see \eqref{eq:ga1}) is generally stronger than $M' \prec_{SV} M$; it plays a crucial role
  in the more general Whitney problem \cite{Bruna80,ChaumatChollet94,Langenbruch94,RainerSchindl16a,Rainer:2021aa}.
  In many important cases the conditions are equivalent.
  There is an explicit weight sequence $S=S(M)$ such that $S \prec_{\ga_1} M$ and $S$ is optimal with respect to this property;
  see \cite{RainerSchindl16a}.
  \item Let $\om$ be a non-quasianalytic weight function.
  If $\om'$ is a suitable other weight function, then $\La^{\{\om'\}} \subseteq j^\infty \cE^{\{\om\}}(\R)$ is equivalent to
  a condition (namely \eqref{eq:omsnq}) purely in terms of $\om',\om$ which we denote by $\om' \prec_{st} \om$.
  There is an explicit weight function $\ka = \ka(\om)$
  such that $\ka \prec_{st} \om$, that is
  $\La^{\{\ka\}} \subseteq j^\infty \cE^{\{\om\}}(\R)$, and $\ka$ is optimal with respect to this property;
  see \cite{BonetMeiseTaylor92}.
  The condition $\om' \prec_{st} \om$ is also crucial in the respective Whitney problem
  \cite{BBMT91,Rainer:2019ac,Rainer:2020aa,Rainer:2020ab,Rainer:2021aa}.
  \item Beside these optimal solutions there is Carleson's solution \cite{Carleson:1961wa} based on a universal moment problem.
  It gives a sequence $Q=Q(M)$ (but is also intimately related with (3)) such that $\La^{\{Q\}} \subseteq j^\infty \cE^{\{M\}}(\R)$.
\end{enumerate}
The assumption that the weight on the right-hand side of \eqref{eq:problem} is non-quasianalytic is no restriction:
it is a necessary consequence of the inclusion \eqref{eq:problem} (for germs at $0$)
if the involved classes strictly contain the real analytic class;
see \cite{RainerSchindl15} and \cite[Section 5]{Rainer:2021aa}.

The use of one-parameter families $\fM$ of weight sequences
admits an efficient unified treatment of Denjoy--Carleman and Braun--Meise--Taylor classes
alike. In fact, for a weight function $\om$ there is a canonical well-behaved family $\fM=\fM_\om$ such that $\La^{\{\om\}} = \La^{\{\fM\}}$
and $\cE^{\{\om\}}(\R) = \cE^{\{\fM\}}(\R)$ as locally convex spaces; cf.\ \cite{RainerSchindl12}.
It also abolishes the borders between the four listed solutions and begs the question about the relationships among them.

In this note we answer this question. Given a suitable family $\fM$ we lift the derived weights $L$, $S$, $\ka$, and $Q$ to the level of families of sequences
$\fL$, $\fS$, $\fK$, and $\fQ$ and clarify the relations among them.
While the construction of $\fL$ and $\fS$ is straightforward,
$\fK$ and $\fQ$ are obtained by a more convoluted procedure involving the interplay between weight sequences $M \in \fM$ and their associated functions $\om_M$,
$\ka_{\om_M}$, and $P_{\om_M}$.
In summary we find, provided that $\fM$ and $\fM'$ satisfy some standard conditions,
\begin{gather}
  \begin{split}
    \label{eq:main1}
    \La^{\{\fS\}} \subseteq \La^{\{\fK\}} = \La^{\{\fQ\}} \subseteq \La^{\{\underline \fL\}} \subseteq \La^{\{\fL\}} \subseteq j^\infty \cE^{\{\fM\}}(\R),
    \\
    \La^{\{\fM'\}} \subseteq j^\infty \cE^{\{\fM\}}(\R)
    \implies \La^{\{\fM'\}} \subseteq \La^{\{\fL\}},
  \end{split}
\end{gather}
underlining the overall optimality of $\fL$ in regard of the mixed Borel problem;
here $\underline \fL$ is obtained from $\fL$ by passing to the log-convex minorants of the sequences in $\fL$.

In the most important Braun--Meise--Taylor case we prove that, for each non-quasianalytic weight function $\om$,
\begin{gather} \label{eq:main2}
    \La^{\{\ka\}} = \La^{\{\fK\}} = \La^{\{\fQ\}} = \La^{\{\ul \fL\}},
\end{gather}
where the families $\fL$, $\fK$, and $\fQ$ are derived from the canonical associated family $\fM_\om$.
So in this case optimality in the mixed Borel problem is achieved by all these different approaches.
In view of the inherent regularity properties of $\ka$ it is not too surprising that we have to pass from $\fL$
which might be quite irregular to  $\ul \fL$.

After recalling background on weights and function and sequence spaces in \Cref{sec:weights}
we review Carleson's solution to the Borel problem in \Cref{sec:Carleson} and adapt it to our setting.
In \Cref{sec:optimal} we prove \eqref{eq:main1} and give sufficient conditions for equality everywhere in \eqref{eq:main1}.
The consequences for the classical cases of Braun--Meise--Taylor and Denjoy--Carleman classes are discussed in \Cref{sec:classical}.
The proof of \eqref{eq:main2} relies on a variant of a result of \cite{BonetMeiseTaylor92}
which we prove in \Cref{appendix}.

\section{Weights} \label{sec:weights}

\subsection{Weight functions}

Let $\omega:[0,\infty)\rightarrow[0,\infty)$ be a continuous increasing function
satisfying $\omega(0)=0$ and $\lim_{t\rightarrow\infty}\omega(t)=\infty$.
We call $\om$ a \emph{pre-weight function} if additionally
\begin{itemize}
  \item $\log(t)=o(\omega(t))$ as $t\rightarrow\infty$,
	\item $\varphi_{\omega}:t\mapsto\omega(e^t)$ is convex.
\end{itemize}
A pre-weight function $\om$ is a \emph{weight function} if it also fulfills
\begin{itemize}
  \item $\omega(2t)=O(\omega(t))$ as $t\rightarrow\infty$.
\end{itemize}

Another important condition is
\begin{equation}
    \label{om6}
    \E H\ge 1 \A t\ge 0 : 2\omega(t)\le\omega(H t)+H.
\end{equation}

Given two pre-weight functions $\om, \si$ we write $\si \preceq \om$ if $\om(t)=O(\si(t))$ as $t \to \infty$
and call $\om$ and $\si$ \emph{equivalent} if $\si \preceq \om$ and $\om \preceq \si$.
Note that $\preceq$ induces a partial order on the set of equivalence classes.
Any equivalence class $[\om]$ contains a pre-weight function $\widetilde \om$ with $\widetilde \om|_{[0,1]}=0$; we say that $\widetilde \om$ is \emph{normalized}.
We will often tacitly assume that this property is satisfied.

Also note that if one representative in $[\om]$ is a weight function then all representatives are weight functions.

A pre-weight function $\om$ is called \emph{non-quasianalytic} if
\[
  \int_0^\infty \frac{\om(t)}{1+t^2}\, dt <\infty
\]
and \emph{quasianalytic} otherwise.
If $\om$ is a non-quasianalytic pre-weight function, then we may consider
\[
  \ka_\om(t) := \int_1^\infty \frac{\om(ts)}{s^2}\, ds = t \int_t^\infty \frac{\om(s)}{s^2}\, ds, \quad t\ge 0,
\]
which turns out to be a concave weight function satisfying $\ka_\om(t)=o(t)$ as $t \to \infty$ and $\om \le \ka_\om$ (which entails $\ka_\om \preceq \om$)
since $\om$ is increasing; it might be quasianalytic.
Note that if $\si$ and $\om$ are equivalent, then so are $\ka_\si$ and $\ka_\om$.

The importance of $\ka_\om$ relies on its \emph{optimality} with respect to the Borel problem \cite{BonetMeiseTaylor92}:
Given that $\om',\om$ are weight functions, $\om$ non-quasianalytic,
$\om'(t) = o(t)$ as $t \to \infty$,
then $\La^{\{\om'\}} \subseteq j^\infty \cE^{\{\om\}}(\R)$
if and only if $\om' \preceq \ka_\om$, i.e.,
\begin{equation} \label{eq:omsnq}
   \E C>0 \A t\ge 0 : \int_1^\infty \frac{\om(ts)}{s^2}\, ds \le C\om'(t) + C;
\end{equation}
we will write $\om' \prec_{st} \om$ for \eqref{eq:omsnq}.
This induces a relation on the equivalence classes of pre-weight functions which is antisymmetric,
since $\om' \prec_{st} \om$ implies $\om' \preceq \om$, and transitive (indeed,
$\om_1 \prec_{st} \om_0$ and $\om_2 \prec_{st} \om_1$ yield
$\om_2 \preceq \ka_{\om_1} \preceq \om_1 \preceq \ka_{\om_0}$ so that $\om_2 \prec_{st} \om_0$).
The pre-weight functions $\om$ that satisfy $\om \prec_{st} \om$, that is $\La^{\{\om\}} = j^\infty \cE^{\{\om\}}(\R)$,
are often called \emph{strong} weight functions; being equivalent to $k_\om$ they actually are weight functions.

For any pre-weight function $\om$ we consider the \emph{Young conjugate} of $\varphi_{\omega}$,
\begin{equation*}
\varphi^{*}_{\omega}(x):=\sup\{x y-\varphi_{\omega}(y): y\ge 0\}, \quad x\ge 0.
\end{equation*}
It is convex, increasing, and satisfies $(\vh_\om^*)^* = \vh_\om$, $\vh_\om^*(t)/t \nearrow \infty$ as $t \to \infty$,
and $\vh_\om^*(0)=0$ (provided that $\om$ is normalized).

\subsection{Weight sequences}

A positive sequence $M = (M_k)_{k \ge 0}$
is called a \emph{weight sequence} if $M_k = \mu_0 \mu_1 \cdots \mu_k$, where
$1 = \mu_0 \le \mu_1 \le \cdots \le \mu_{k-1} \le \mu_k \nearrow \infty$.
That $\mu$ is increasing amounts to log-convexity of $M$.
We have $M_k^{1/k} \nearrow \infty$.
If even $\mu_k/k$ is increasing, we say that $M$ is \emph{strongly log-convex}.

For two positive sequences $M,N$ we write $M \preceq N$ if
\[
  \sup_{k\ge 1}\Big(\frac{M_k}{N_k}\Big)^{1/k}<\infty,
\]
and we call $M$ and $N$ \emph{equivalent}, if
$M \preceq N$ and $N \preceq M$.
The relation $M \preceq N$ induces a partial order on the set of equivalence classes.

Note that if $M$ is log-convex and $\mu_k \nearrow \infty$ then the equivalence class $[M]$
contains a weight sequence $\widetilde M$. That means $\widetilde M$ also satisfies $1= \widetilde M_0 \le \widetilde M_1$.

A weight sequence $M$ is called \emph{non-quasianalytic,} if
\begin{equation*}
	\sum_{k=1}^{\infty}\frac{1}{\mu_k}<\infty
\end{equation*}
and \emph{quasianalytic} otherwise.
We say that $M$ has \emph{moderate growth} if
\begin{equation*}
	\exists C\ge 1 \;\forall j,k\in\NN : M_{j+k}\le C^{j+k} M_j M_k.
\end{equation*}

\subsection{Associated function}

With a positive sequence $M$ satisfying $M_k^{1/k}\to \infty$
we associate (cf.\ \cite[Chapitre I]{mandelbrojtbook} and \cite[Definition 3.1]{Komatsu73})
the function $\omega_M: [0,\infty) \to [0,\infty)$ defined by
\begin{equation*}\label{assofunc}
\omega_M(t):= \sup_{k\in\NN}\log\Big(\frac{t^k M_0}{M_k}\Big), \quad \text{ for } t>0, \quad \omega_M(0):=0.
\end{equation*}

\begin{lemma}[{Cf.\ \cite[Lemma 2.4]{sectorialextensions} and \cite[Lemma 3.1]{sectorialextensions1}}]\label{assofuncproper}
For a weight sequence $M$:
\begin{enumerate}
  \item $\omega_M$ is a pre-weight function.
  \item $M$ has moderate growth if and only if $\om_M$ satisfies \eqref{om6}.
  \item $(M_k/k!)^{1/k}\to \infty$ if and only if $\om_M(t) = o(t)$ as $t \to \infty$.
\end{enumerate}
\end{lemma}

The \emph{log-convex minorant} $\ul M$ of a positive sequence $M$ satisfying $M_k^{1/k}\to \infty$ is given by
\begin{equation} \label{eq:minorant}
  \ul M_k := M_0 \sup_{t \ge 0} \frac{t^k}{e^{\om_M(t)}},\quad  k \in \N.
\end{equation}
We have $\ul M_k^{1/k}\to \infty$ and
if $L \le M$ is log-convex then $L \le \ul M \le M$.

\subsection{Weight matrices}

Cf.\ \cite[Section 4]{RainerSchindl12}.
A \emph{weight matrix} $\fM =\{M^{(x)}: x\in \RR_{>0}\}$ is a one parameter family of weight sequences $M^{(x)}$ such that $M^{(x)} \le M^{(y)}$ if $x \le y$.
Weight matrices are a convenient technical tool for working with weight functions:

\begin{lemma}[{\cite[Section 5]{RainerSchindl12}}]\label{omegaproperties}
	With every normalized pre-weight function $\om$ one can associate a weight matrix
	$\fM= \fM_\om :=\{M^{(x)}: x>0\}$
	by setting
  \begin{equation} \label{eq:asswm}
    M^{(x)}_k:=\exp\Big(\frac{1}{x}\varphi^{*}_{\omega}(x k)\Big).
  \end{equation}
  $\om$ is non-quasianalytic if and only if some/each $M^{(x)}$ is non-quasianalytic.
  All weight sequences $M^{(x)}$ are equivalent if and only if $\om$ satisfies \eqref{om6}
  which in turn is equivalent to some/each $M^{(x)}$ having moderate growth.

  If $\om$ is even a weight function, then
     \begin{equation}\label{newexpabsorb}
     \A h\ge 1 \E A\ge 1 \A x>0 \E D\ge 1 \A k\in\NN :  h^k M^{(x)}_k \le D M^{(Ax)}_k.
     \end{equation}

     If $\om$ is not normalized, we still define $\fM_\om$ by \eqref{eq:asswm}, but
     $1 = M^{(x)}_0 \le M^{(x)}_1$ might fail.
\end{lemma}

A weight matrix $\fM$ is called \emph{non-quasianalytic} if all $M \in \fM$ are non-quasianalytic.

Let $\fM$ and $\fN$ be weight matrices. We write $\fM \{\preceq\} \fN$ if
\[
    \A M \in \fM \E N \in \fN : M \preceq N
\]
and say that $\fM$ and $\fN$ are \emph{R-equivalent} if $\fM \{\preceq\} \fN$ and $\fN \{\preceq\} \fM$.
Note that if there exists a non-quasianalytic $M \in \fM$, then there is a non-quasianalytic weight matrix
R-equivalent to $\fM$.

A weight matrix $\fM$ is said to have \emph{R-moderate growth} if
\begin{equation*}
   \A M \in \fM \E N \in \fM \E C\ge 1 \A j,k \in \N : M_{j+k} \le C^{j+k} N_j N_k.
\end{equation*}
Note that $\fM_\om$ has R-moderate growth; see \cite{RainerSchindl12}.

We remark that the prefix ``R-" and the brackets ``$\{\cdot\}$'' indicate that the notions are tied to the Roumieu case;
they have Beurling-relatives which will not be discussed in this paper.

\subsection{Function and sequence spaces}

Let $M$ be a positive sequence.
For $\si > 0$ and $n \in \N$ we define the Banach space
\[
  \cE^M_\si([-n,n])
  :=\Big\{ f \in C^\infty([-n,n]):
  \sup_{x \in [-n,n],\, k \in \N}  \frac{|f^{(k)} (x)|}{\si^{k}M_{k}}<\infty \Big\}
\]
and the \emph{Denjoy--Carleman classes of Roumieu type}
\[
\cE^{\{M\}}(\R) :=  \on{proj}_{n \in \N} \on{ind}_{\si  \in \N} \cE^{M}_\si([-n,n]).
\]
Let $\fM$ be a directed family of positive sequences.
We set
\begin{align*}
  \cE^{\{\fM\}}(\R) := \on{proj}_{n \in  \N} \on{ind}_{\si  \in \N} \on{ind}_{M \in \fM} \cE^{M}_\si([-n,n])
\end{align*}
Let $\om$ be a normalized pre-weight function. We define
\[
\cE^{\om}_\si([-n,n]):= \Big\{f \in C^\infty([-n,n]):
\sup_{x \in [-n,n],\, k \in \N} \frac{|f^{(k)}(x)|}{e^{\frac{1}{\si} \vh_\om^* (\si k)}}<\infty \Big\}
\]
and the \emph{Braun--Meise--Taylor class of Roumieu type}
\[
\cE^{\{\om\}}(\R):= \on{proj}_{n \in \N} \on{ind}_{\si \in \N} \cE^\om_\si([-n,n]).
\]
The corresponding sequence spaces are defined as follows:
\begin{gather*}
  \La^M_\si
  :=\Big\{ a \in \C^\N :
  \sup_{k \in \N}  \frac{|a_k|}{\si^{k}M_{k}}<\infty \Big\},
  \quad \La^{\om}_\si := \Big\{ a \in \C^\N :
  \sup_{k \in \N} \frac{|a_k|}{e^{\frac{1}{\si} \vh_\om^* (\si k)}}<\infty \Big\}
  \\
  \La^{\{M\}} := \on{ind}_{\si \in \N} \La^M_\si, \quad
  \La^{\{\fM\}} := \on{ind}_{M \in \fM} \La^{\{M\}}, \quad \La^{\{\om\}} := \on{ind}_{\si \in \N} \La^\om_\si.
\end{gather*}
By \Cref{omegaproperties} (and \cite{RainerSchindl12}), we always have $\cE^{\{\om\}}(\R) \subseteq \cE^{\{\fM_\om\}}(\R)$ and
$\La^{\{\om\}} \subseteq \La^{\{\fM_\om\}}$ and
topological isomorphisms $\cE^{\{\om\}}(\R) = \cE^{\{\fM_\om\}}(\R)$
and $\La^{\{\om\}} = \La^{\{\fM_\om\}}$ if $\om$ is even a weight function.

The order relations $\preceq$ and $\{\preceq\}$ on the weights reflect inclusion relations of the corresponding function and sequence spaces.
In particular, for positive sequences $M',M$ we have
\begin{align*}
  M' \preceq M \quad\Longleftrightarrow&\quad \La^{\{M'\}} \subseteq \La^{\{M\}},
  \\
  M' \preceq M \quad\Longrightarrow&\quad \cE^{\{M'\}}(\R) \subseteq \cE^{\{M\}}(\R),
  \intertext{and provided that $M'$ is a weight sequence}
  M' \preceq M \quad\Longleftarrow\,&\quad \cE^{\{M'\}}(\R) \subseteq \cE^{\{M\}}(\R).
\end{align*}
So equivalent weights determine the same function and sequence space.
For details we refer to \cite{RainerSchindl12} and \cite{Rainer:2021aa}.

\section{On Carleson's solution of the Borel problem} \label{sec:Carleson}

\subsection{A moment problem} \label{sec:moment}
Let $\om$ be a non-quasianalytic pre-weight function which we extend to $\R$ by setting $\om(t) := \om(|t|)$.
(It is easy to check that $W := \exp \o \om$ satisfies the assumptions in \cite{Carleson:1961wa}.)

The harmonic extension
\[
  P_\om(x+iy) := \begin{cases}
  \frac{|y|}{\pi} \int_{-\infty}^\infty \frac{\om(t)}{(x-t)^2+y^2}\, dt & \text{ if } y\ne 0
  \\
  \om(x) & \text{ if } y = 0
\end{cases}
\]
is continuous on $\C$ and
harmonic in the upper and lower halfplane.
By \cite[(2.3)]{Carleson:1961wa} (see also \cite[Lemma 3.3]{BonetMeiseTaylor92}),
\begin{align*}
  \max_{x^2 + y^2 \le r^2} P_\om(x+iy) &\le P_\om(i r) = \frac{r}{\pi} \int_{-\infty}^\infty \frac{\om(t)}{t^2+r^2}\, dt
  = \frac{2}{\pi} \int_0^\infty \frac{\om(rs)}{1+s^2}\, ds
  \\
  &\le  \frac{2}{\pi} \Big( \om(r) +  \int_1^\infty \frac{\om(rs)}{s^2}\, ds \Big)
  = \frac{2}{\pi} (\om(r) + \ka_\om(r)) \le \frac{4}{\pi} \ka_\om(r).
\end{align*}
On the other hand,
\begin{align*}
   \ka_\om(r) = \int_1^\infty \frac{\om(rs)}{s^2}\, ds
   \le 2 \int_1^\infty \frac{\om(rs)}{1+s^2}\, ds
   \le 2 \int_0^\infty \frac{\om(rs)}{1+s^2}\, ds
   =\pi P_\om(ir).
\end{align*}
We see that
\begin{equation} \label{lem:BMT}
   P_\om(i r) \le \frac{4}{\pi} \ka_\om(r) \le 4  P_\om(i r), \quad r>0.
\end{equation}
We define
\begin{equation} \label{eq:Q}
  Q_k := \sup_{r>0} \frac{r^{k+\frac{1}{2}}}{e^{\frac{1}{2}P_\om(ir)}}, \quad  k \in \N.
\end{equation}

Let $\cF_\om$ be the space of measurable functions $f : \R \to \R$ such that
\[
  \|f\|_\om^2 := \int_{-\infty}^\infty |f(t)|^2 e^{-\om(t)}\, dt < \infty.
\]
Fix a positive real sequence $\la = (\la_k)_{k \in \N}$ and let $\cS_\la$ be the space of sequences $s = (s_k)_{k \in \N}$
such that
\[
  \|s\|_\la^2 := \sum_{k=0}^\infty \frac{|s_k|^2}{\la_k^2} < \infty.
\]

\begin{theorem}[{\cite[Theorem 1]{Carleson:1961wa}}] \label{thm:Carleson1}
  For each $s \in \cS_\la$ there exists $f \in \cF_\om$ with
  \[
    L_k(f) := \int_{-\infty}^\infty f(t)  t^k e^{-\om(t)} \, dt = s_k, \quad k \in \N,
  \]
  provided that
  \begin{equation*}
     \sum_{k=0}^\infty \Big(\frac{\la_k}{Q_k}\Big)^2 < \infty.
  \end{equation*}
\end{theorem}

\subsection{A solution of the Borel problem} \label{sec:wmK}

Let $M$ be a non-quasianalytic weight sequence and $\om_M$ the associated function.
Then
\[
\widetilde \om_M(t) := \om_M (t) + \log(1+t^2)
\]
is a non-quasianalytic pre-weight function that is equivalent to $\om_M$.

Let $\fM = \{M^{(\al)} : \al >0\}$ be a non-quasianalytic weight matrix.
For $\al > 0$ and $j \in \N$ set
\[
  \ka_\al := \ka_{\widetilde \om_{M^{(\al)}}} \quad \text{ and } \quad K^{(\al)}_j := \exp(\vh^*_{\ka_\al}(j)).
\]
Then $K^{(\al)} = (K^{(\al)}_j)$ is a weight sequence by the properties of the Young conjugate
and $\fK = \fK(\fM) := \{K^{(\al)} : \al >0\}$ is a weight matrix.
(Strictly speaking, we take a normalized representative in the equivalence class of $\ka_\al$ in order to have $1=K^{(\al)}_0 \le K^{(\al)}_1$.)
On the other hand we have the collection $\fQ = \fQ(\fM) :=\{Q^{(\al)} : \al>0\}$, where
$Q^{(\al)}$ is the sequence defined in \eqref{eq:Q} with $P_\om$ replaced by
\[
P_\al :=  P_{\widetilde \om_{M^{(\al)}}}.
\]

\begin{lemma} \label{lem:fKproperties}
Let $\fM$ be a non-quasianalytic weight matrix.
Then $\fK = \fK(\fM)$ is a weight matrix such that $(K^{(\al)}_j/j!)^{1/j}\to \infty$ and $K^{(\al)}_j/M^{(\al)}_j$ is bounded for all $\al>0$.
\end{lemma}

\begin{proof}
   That $\fK$ is a weight matrix is an easy consequence of the definitions.
   Now $\om_{K^{(\al)}}$ is equivalent to $\ka_\al$, by \cite[Lemma 5.7]{RainerSchindl12}, and $\ka_\al(t)=o(t)$ as $t\to \infty$
   which shows $(K^{(\al)}_j/j!)^{1/j}\to \infty$, by \Cref{assofuncproper}.
  Finally, $\om_{M^{(\al)}} \le \widetilde \om_{M^{(\al)}} \le \ka_\al$ implies $K^{(\al)}\le M^{(\al)}$ in view of \eqref{eq:minorant}.
\end{proof}

\begin{lemma} \label{lem:mgK}
  Let $\fM$ be a non-quasianalytic weight matrix of R-moderate growth.
  Then $\fK = \fK(\fM)$ has R-moderate growth
  which is equivalent to
  \begin{equation} \label{eq:asska}
    \A \al>0 \E \be>0 \E H\ge 1 \A t\ge 0 : 2 \ka_\be (t) \le \ka_\al(Ht) +H.
  \end{equation}
\end{lemma}

\begin{proof}
   That $\fM$ has R-moderate growth is, by \cite[Proposition 3.6]{testfunctioncharacterization}, equivalent to
   \begin{equation} \label{eq:asska1}
     \A \al>0 \E \be>0 \E H\ge 1 \A t\ge 0 : 2 \om_{M^{(\be)}} (t) \le \om_{M^{(\al)}}(Ht) +H.
   \end{equation}
   Since $\widetilde \om_{M^{(\al)}}$ is equivalent to $\om_{M^{(\al)}}$, this implies by iteration that
   \[
      \A \al>0 \E \be>0 \E H\ge 1 \A t\ge 0 : 2 \widetilde \om_{M^{(\be)}} (t) \le \widetilde \om_{M^{(\al)}}(Ht) +H.
   \]
   Thus \eqref{eq:asska} holds:
   \begin{align*}
      2 \ka_\be(r) = 2  \int_1^\infty \frac{\widetilde \om_{M^{(\be)}}(rt)}{t^2} \,dt
      \le \int_1^\infty \frac{\widetilde \om_{M^{(\al)}}(Hrt)}{t^2} \,dt +  H \int_1^\infty \frac{1}{t^2} \,dt = \ka_\al(Hr) +  H.
   \end{align*}
   To see that \eqref{eq:asska} is equivalent to $\fK$ having R-moderate growth
   observe that, for each $\al>0$,
   $\ka_\al$ is equivalent to $\om_{K^{(\al)}}$, by \cite[Lemma 5.7]{RainerSchindl12},
   and use the remark at the beginning of the proof.
\end{proof}

\begin{proposition} \label{prop:KW}
  Let $\fM$ be a non-quasianalytic weight matrix of R-moderate growth.
  Then $\fK = \fK(\fM)$ and $\fQ = \fQ(\fM)$ are R-equivalent.
\end{proposition}

\begin{proof}
Let $\al>0$ be fixed. By \eqref{lem:BMT} and \eqref{eq:asska},
there are $\be>0$ and $H\ge 1$ such that
\begin{align*}
   \frac{1}{2} P_{\be}(it) \le 2 \ka_\be(t) \le \ka_\al(Ht) +H, \quad t>0.
\end{align*}
Then, using without loss of generality $\ka_\al|_{[0,1]} =0$,
\begin{multline*}
   Q^{(\be)}_n = \sup_{r>0} \frac{r^{n+\frac{1}{2}}}{e^{\frac{1}{2}P_{\be}(ir)}}
   \ge \frac{1}{e^H} \sup_{r>0} \frac{r^{n+\frac{1}{2}}}{e^{\ka_{\al}(Hr)}}
   = \frac{1}{H^{n+\frac{1}{2}} e^H} \sup_{s>0} \frac{s^{n+\frac{1}{2}}}{e^{\ka_{\al}(s)}}
   \\
   = \frac{1}{H^{n+\frac{1}{2}} e^H} \sup_{s\ge 1} \frac{s^{n+\frac{1}{2}}}{e^{\ka_{\al}(s)}}
   \ge \frac{1}{H^{n+\frac{1}{2}} e^H} \sup_{s\ge 1} \frac{s^{n}}{e^{\ka_{\al}(s)}}
   = \frac{1}{H^{n+\frac{1}{2}} e^H} K^{(\al)}_n,
\end{multline*}
that is $K^{(\al)} \preceq Q^{(\be)}$.
Again by \eqref{lem:BMT} and \eqref{eq:asska},
there are $\be>0$ and $H\ge 1$ such that
\begin{align*}
  \ka_\be(t) \le \frac{1}{8} \ka_\al(Ht) + \frac{H}8 \le \frac{1}{2} P_\al(iHt) + \frac{H}8, \quad t>0.
\end{align*}
Thus, using $\ka_\be|_{[0,1]}=0$,
\begin{multline*}
   Q^{(\al)}_n = H^{n+\frac{1}{2}} \sup_{r>0} \frac{r^{n+\frac{1}{2}}}{e^{\frac{1}{2} P_{\al}(iHr)}}
   \le H^{n+\frac{1}{2}} e^{\frac{H}8} \sup_{r>0} \frac{r^{n+\frac{1}{2}}}{e^{\ka_{\be}(r)}}
   = H^{n+\frac{1}{2}} e^{\frac{H}8} \sup_{r\ge 1} \frac{r^{n+\frac{1}{2}}}{e^{\ka_{\be}(r)}}
   \\
   \le H^{n+\frac{1}{2}} e^{\frac{H}8} \sup_{r\ge 1} \frac{r^{n+1}}{e^{\ka_{\be}(r)}}
   = H^{n+\frac{1}{2}} e^{\frac{H}8} K^{(\be)}_{n+1}.
\end{multline*}
Since $\fK$ has R-moderate growth we are done.
\end{proof}

The following theorem is a simple generalization of \cite[Theorem 2]{Carleson:1961wa}.

\begin{theorem} \label{thm:moment}
  Let $\fM$ be a non-quasianalytic weight matrix of R-moderate growth and consider $\fK = \fK(\fM)$ and $\fQ = \fQ(\fM)$.
  Then $\La^{\{\fK\}} = \La^{\{\fQ\}}  \subseteq j^\infty \cE^{\{\fM\}}(\R)$.
\end{theorem}

\begin{proof}
   The identity $\La^{\{\fK\}} = \La^{\{\fQ\}}$ is a consequence of \Cref{prop:KW}.
   Let $a \in \La^{\{\fQ\}}$, i.e., there exist $\al,C,\rh >0$ such that $|a_n| \le C \rh^n Q^{(\al)}_n$ for all $n$.
   Set $s_n := (3 \rh i)^{-n} a_n$ and $\la_n := 2^{-n} Q^{(\al)}_n$.
   By \Cref{thm:Carleson1}, there exists $f \in \cF_{\widetilde \om_{M^{(\al)}}}$ such that $L_n(f)=s_n$ for all $n$.
   The function
   \[
    g(t) := \int_{-\infty}^\infty e^{3\rh i xt} f(x) e^{-\widetilde \om_{M^{(\al)}}(x)} \,dx
   \]
   fulfills $j^\infty g = a$ and belongs to $\cE^{\{\fM\}}(\R)$. Indeed,
   \begin{align*}
     |g^{(n)}(t)| &\le (3\rh)^n \int_{-\infty}^\infty |x|^n |f(x)| e^{-\widetilde \om_{M^{(\al)}}(x)} \,dx
     \\
     &\le (3\rh)^n \|f\|_{\widetilde \om_{M^{(\al)}}}
     \Big(\int_{-\infty}^\infty x^{2n} e^{-\widetilde \om_{M^{(\al)}}(x)} \,dx \Big)^{1/2}
   \end{align*}
   and
   \begin{align*}
      x^{2n} e^{-\widetilde \om_{M^{(\al)}}(x)} = \frac{x^{2n}}{\exp(\om_{M^{(\al)}}(x))} \frac{1}{1+x^2} \le  \frac{1}{1+x^2} \sup_{x\ge 0} \frac{x^{2n}}{\exp(\om_{M^{(\al)}}(x))} =  \frac{M^{(\al)}_{2n}}{1+x^2}.
   \end{align*}
   Since $\fM$ has R-moderate growth,
   there exist $\be >0$ and $H\ge 1$ such that
   $M^{(\al)}_{2n} \le H^n (M^{(\be)}_{n})^2$ for all $n$. Thus
   \begin{align*}
      |g^{(n)}(t)| \le (3\rh \sqrt{H})^n \|f\|_{\widetilde \om_{M^{(\al)}}} M^{(\be)}_{n}  \Big(\int_{-\infty}^\infty \frac{1}{1+x^2} \,dx \Big)^{1/2},
   \end{align*}
   that is $g \in \cE^{\{\fM\}}(\R)$.
\end{proof}

\section{Other derived weights related to the Borel problem} \label{sec:optimal}

\subsection{Optimal solution of the Borel problem}

Let $M$ be a non-quasianalytic weight sequence and $M'$ a positive sequence.
The condition that gives the optimal solution in the Borel problem is
\begin{equation} \label{eq:SV}
   \E s \in \N_{\ge 1} : \sup_{j \in \N_{\ge 1}} \frac{\sup_{0\le i <j} \big(\frac{M'_j}{s^j M_i}\big)^{1/(j-i)}}{j} \sum_{k \ge j} \frac{1}{\mu_k} < \infty
\end{equation}
which we abbreviate by $M' \prec_{SV} M$:

\begin{theorem}[{\cite{surjectivity}, \cite[Theorem 3.2]{mixedramisurj}, \cite[Theorem 2.2]{maximal}}] \label{thm:DCSV}
  Let $M$ be a non-quasianalytic weight sequence and $M'$ a positive sequence
  such that $\liminf_{k \to \infty} (M'_k/k!)^{1/k} >0$. Then
  $\La^{\{M'\}} \subseteq j^\infty \cE^{\{M\}}(\R)$
  if and only if $M' \prec_{SV} M$.
\end{theorem}

In the listed references also the assumption $M' \preceq M$ (or even a stronger assumption)
is made. But
note that, by \cite[Lemma 3.2]{maximal},
\begin{equation} \label{eq:SVimpliesincl}
    M'\prec_{SV} M \implies M' \preceq M,
\end{equation}
and clearly also $\La^{\{M'\}} \subseteq j^\infty \cE^{\{M\}}(\R)$ implies $M' \preceq M$
so that \Cref{thm:DCSV} holds as stated without the additional assumption $M'\preceq M$.

We remark that $\prec_{SV}$ induces a relation on the set of equivalence classes of weight sequences $M$ such that $\liminf_{k \to \infty} (M_k/k!)^{1/k} >0$
which is
antisymmetric and transitive (indeed, $M_1 \prec_{SV} M_0$ and $M_2 \prec_{SV} M_1$ imply
$\La^{\{M_2\}} \subseteq j^\infty\cE^{\{M_1\}}(\R) \subseteq j^\infty\cE^{\{M_0\}}(\R)$ by \eqref{eq:SVimpliesincl} and so $M_2 \prec_{SV} M_0$ by \Cref{thm:DCSV}).

\begin{theorem} \label{thm:SVchar}
  Let $\fM$ be a non-quasianalytic weight matrix and
  $\fM'$ a one-parameter family of positive sequences such that
  $\liminf_{k \to \infty} (M'_k/k!)^{1/k} >0$ for all $M' \in \fM'$.
  Then the following conditions are equivalent:
  \begin{enumerate}
    \item $\La^{\{\fM'\}} \subseteq j^\infty \cE^{\{\fM\}}(\R)$.
    \item $\forall M' \in \fM' \E M \in \fM :  M' \prec_{SV} M$.
  \end{enumerate}
\end{theorem}

\begin{proof}
  \Cref{thm:DCSV} yields that (2) is equivalent to
  \begin{enumerate}
    \item[(3)] $\forall M' \in \fM' \E M \in \fM : \La^{\{M'\}} \subseteq j^\infty \cE^{\{M\}}(\R)$.
  \end{enumerate}

  That (3) implies (1) is clear.
  To see that (1) implies (3) let $M' \in \fM'$ and note that we may assume that $\fM = \{M^{(n)} : n \in \N_{\ge 1}\}$.
  By (1), $\La^{M'}_1 \subseteq j^\infty \cD^{\{\fM\}}([-1,1])$
  (by multiplication with a suitable cutoff function),
  where
  \[
    \cD^{\{\fM\}}([-1,1]) := \on{ind}_{n \in \N_{\ge 1}} \cD^{M^{(n)}}_n ([-1,1])
  \]
  and
  $\cD^{M^{(n)}}_n ([-1,1]) := \{f \in \cE^{M^{(n)}}_n([-1,1]) : \on{supp}(f) \subseteq [-1,1]\}$.
  By Grothendieck's factorization theorem \cite[24.33]{MeiseVogt97},
  $\La^{M'}_1 \subseteq j^\infty \cD^{M^{(n)}}_n ([-1,1])$ for some $n$.
  This inclusion implies $\La^{\{M'\}} \subseteq j^\infty \cE^{\{M^{(n)}\}} (\R)$ and hence (3) is proved.
  Indeed, if $a=(a_k) \in \La^{\{M'\}}$ then $a \in \La^{M'}_\rh$ for some $\rh>0$ and hence
  $b := (\rh^{-k}a_k) \in \La^{M'}_1$. There exists $g \in \cD^{M^{(n)}}_n ([-1,1])$ with $b = j^\infty g$
  so that $f(x):= g(\rh x) \in \cE^{\{M^{(n)}\}} (\R)$ satisfies $a = j^\infty f$.
\end{proof}

\begin{corollary} \label{cor:SVchar}
  Let $\fM$ be a non-quasianalytic weight matrix of R-moderate growth and consider $\fK = \fK(\fM)$.
  Then $\forall \al>0 \E \be>0 : K^{(\al)} \prec_{SV} M^{(\be)}$.
\end{corollary}

\begin{proof}
  This follows from \Cref{lem:fKproperties}, \Cref{thm:moment}, and \Cref{thm:SVchar}.
\end{proof}

\subsection{The derived sequence $L$} \label{sec:L}

Let $M$ be a non-quasianalytic weight sequence.
We define the sequence $L = L(M)$ by setting
\begin{equation} \label{eq:L}
   L_k := \min_{0 \le j < k} \Big( \Big(\frac{k}{\sum_{\ell \ge k} \mu_\ell^{-1}} \Big)^{k-j} M_j \Big), \quad k\ge1, \quad L_0 := 1.
\end{equation}
The importance of $L$ relies on its \emph{optimality} with respect to $\cdot \prec_{SV} M$ (cf.\ \cite[Theorem 3.3]{maximal}):
We have $L \prec_{SV} M$ and if
$M'$ is a positive sequence with $\liminf_{k \to \infty} (M'_k/k!)^{1/k} >0$ and
$M' \prec_{SV} M$ then $M' \preceq L \preceq M$.
If $M'$ is log-convex we also have $M' \preceq \ul L \preceq M$, where
$\ul L$ is the log-convex minorant of $L$.
We remark that $(L_k/k!)^{1/k} \to \infty$; see \cite[Lemma 3.2]{maximal}.

\begin{lemma} \label{lem:orderL}
  If $M \preceq N$ are two non-quasianalytic weight sequences, then
  $L(M) \preceq L(N)$ and $\ul L(M) \preceq \ul L(N)$.
\end{lemma}

\begin{proof}
   By \Cref{thm:DCSV}, $L(M) \prec_{SV} M$ implies $L(M) \prec_{SV} N$ and so $L(M) \preceq L(N)$ by the optimality.
   Passing to the log-convex minorant preserves the order relation, see \cite[Lemma 2.6]{RainerSchindl12}.
\end{proof}

\subsection{The derived sequence $S$} \label{descendant}

Let $M$ be a non-quasianalytic weight sequence and $M'$ a positive sequence.
We consider the relation $M' \prec_{\ga_1} M$ defined by
\begin{equation} \label{eq:ga1}
    \sup_{j \in \N_{\ge 1}} \frac{\mu'_j}{j} \sum_{k \ge j} \frac{1}{\mu_k} < \infty.
\end{equation}
Let us recall a construction (see \cite[Section 4.1]{RainerSchindl16a}, \cite[Remark 9]{mixedsectorialextensions}, and
also \cite{Petzsche88})
which yields a weight sequence
that is \emph{optimal} with respect to $\cdot \prec_{\ga_1} M$ in the following sense.
The strongly log-convex weight sequence $S = S(M)$ defined by $S_k = \sigma_0 \sigma_1 \cdots \sigma_k$, where
$\sigma_0 := 1$ and
\[
\sigma_k := \tau_1 \frac{ k}{\tau_k},
\quad \tau_k := \frac{k}{\mu_k}+\sum_{\ell\ge k}\frac{1}{\mu_\ell},\quad k\ge 1,
\]
satisfies $\si \lesssim \mu$, $S \prec_{\ga_1} M$,
and if $M'$ is another weight sequence satisfying $\mu' \lesssim \mu$ and $M' \prec_{\ga_1} M$
then $\mu' \lesssim \si$;
see \cite[Lemma 4.2]{RainerSchindl16a}.
Note that $\si \lesssim \mu$ means that $\si/\mu$ is bounded and implies $S \preceq M$.

We have $S(M)\preceq L(M)$, since $M' \prec_{\ga_1} M$ implies $M' \prec_{SV} M$; cf.\ \cite[Lemma 2.4]{mixedramisurj} and \cite{surjectivity}.

Suppose that also $M'$ is a weight sequence and $\mu' \le \mu$.
Then $M' \prec_{\ga_1} M$ implies $\om_{M'} \prec_{st} \om_M$,
  see \cite[Lemma 5.7]{Rainer:2019ac}.
  If $M'$ additionally has moderate growth, then
  \[
    M' \prec_{SV} M \quad\Longleftrightarrow\quad M' \prec_{\ga_1} M \quad\Longleftrightarrow\quad \omega_{M'} \prec_{st} \omega_M;
  \]
  cf.\ \cite[Lemma 5.8]{mixedramisurj} and \cite[Remark 2.1]{maximal}.
  So $\prec_{\ga_1}$ induces a relation on equivalence classes of positive sequences which is always antisymmetric
  and becomes transitive if we restrict to weight sequences of moderate growth.

  In light of this it is important to know under which circumstances the derived sequence $S$ has moderate growth.
  This is the case if $M$ has moderate growth and also under the weaker condition
  \begin{equation} \label{eq:Mmg}
     \liminf_{j \to \infty} \frac{\mu_j}{j} \sum_{k \ge 2j} \frac{1}{\mu_k} >0;
  \end{equation}
  see \cite[Lemma 6]{mixedsectorialextensions}.
  In that case $S$ and $L$ are equivalent, see \cite[Theorem 3.11]{maximal}.

\subsection{Relations among the derived sequences}

Let $\fM = \{M^{(\al)}: \al>0\}$ be a non-quasianalytic weight matrix
and $\fK = \{K^{(\al)}: \al>0\}$ the derived weight matrix from \Cref{sec:wmK}.
For each $M^{(\al)}$ we consider the derived sequences $S^{(\al)}$, $L^{(\al)}$, and $\ul L^{(\al)}$ (the log-convex minorant of $L^{(\al)}$)
and the families
\[
  \fS = \{S^{(\al)} : \al >0\}, \quad \fL = \{L^{(\al)} : \al >0\}, \quad \ul \fL = \{\ul L^{(\al)} : \al >0\}.
\]
Formally, these collections are not weight matrices as defined above, but their deficiencies are minor and carry no weight.
We have $S^{(\al)} \preceq \ul L^{(\al)} \le L^{(\al)}$ for each $\al>0$.

\begin{theorem} \label{thm:SKL}
  Let $\fM$ be a non-quasianalytic weight matrix of R-moderate growth.
  Then the derived families satisfy $\fS \{\preceq\} \fK \{\preceq\}\ul \fL \{\preceq\} \fL$.
\end{theorem}

\begin{proof}
   Let us first show $\fK \{\preceq\} \ul \fL$.
   By \Cref{cor:SVchar}, for each $\al>0$ there is $\be>0$ such that $K^{(\al)} \prec_{SV} M^{(\be)}$
   and thus $K^{(\al)} \preceq L^{(\be)}$ by optimality.
   Since $K^{(\al)}$ is log-convex, we also have $K^{(\al)} \preceq \ul L^{(\be)}$.

   For $\fS \{\preceq\} \fK$ observe that for each $\al>0$ we may assume that $\si^{(\al)} \le \mu^{(\al)}$ by dividing
   $\si^{(\al)}$ by a suitable constant.
   Then $\om_{S^{(\al)}} \prec_{st} \om_{M^{(\al)}}$ holds (cf.\ \Cref{descendant}), i.e., $\om_{S^{(\al)}} \preceq \ka_\al$, since $\ka_\al$ and
   $\ka_{\om_{M^{(\al)}}}$ are equivalent.
   By \cite[Lemma 5.7]{RainerSchindl12}, $\om_{K^{(\al)}}$ is equivalent to $\ka_\al$
   so that $\om_{K^{(\al)}} \le  C \om_{S^{(\al)}} + C$ for some positive integer $C$.
   Then
   \begin{align*}
      S^{(\al)}_n
      = \sup_{t\ge 0} \frac{t^n}{\exp(\om_{S^{(\al)}}(t))}
      &\le e \sup_{t\ge 0} \frac{t^n}{\exp(C^{-1}\om_{K^{(\al)}}(t))}
      \\
      &= e \Big(\sup_{t\ge 0} \frac{t^{Cn}}{\exp(\om_{K^{(\al)}}(t))}\Big)^{1/C}
      = e (K^{(\al)}_{Cn})^{1/C}.
   \end{align*}
   There exists $\be>0$ such that $S^{(\al)} \preceq K^{(\be)}$,
   since $\fK$ has R-moderate growth, by \Cref{lem:mgK}.
   Since $\ul \fL \{\preceq\} \fL$ is obvious, we are done.
\end{proof}

We may now complete the proof of \eqref{eq:main1}:
Let $\fM$ be a non-quasianalytic weight matrix of R-moderate growth.
Then \Cref{prop:KW}, \Cref{thm:DCSV}, \Cref{thm:SVchar}, and \Cref{thm:SKL} yield the sequence of inclusions in
the first line of \eqref{eq:main1}.

Suppose that $\fM'$ is a one-parameter family of positive sequences such that
$\liminf_{k \to \infty} (M'_k/k!)^{1/k} >0$ for all $M' \in \fM'$
and
$\La^{\{\fM'\}} \subseteq j^\infty \cE^{\{\fM\}}(\R)$.
Then for each $M' \in \fM'$ there is $M \in \fM$ such that $M' \prec_{SV} M$, by \Cref{thm:SVchar},
and hence $M' \preceq L(M)$, by optimality. That means $\fM' \{\preceq\} \fL$ and the second line in \eqref{eq:main1} is proved.

\subsection{Sufficient conditions for R-equivalence of the derived families $\fS$, $\fK$ and $\fL$}

\begin{theorem} \label{thm:sufficientRequi}
    Let $\fM = \{M^{(\al)}: \al>0\}$ be a non-quasianalytic weight matrix of R-moderate growth
     such that
     $\mu^{(\al)} \le \mu^{(\be)}$ if $\al \le \be$
     and
     \begin{equation} \label{eq:liminf}
        \A \al>0 \E \be>0 : \liminf_{k \to \infty} \frac{\mu^{(\be)}_k}{k} \sum_{j \ge k} \frac{1}{\mu^{(\al)}_j}>0.
     \end{equation}
    Assume that the derived family $\fS$ has the property
    \begin{equation} \label{eq:roquS}
       \A \al>0 \E \be>0 \E A\ge 1 \A j \in \N_{\ge 1} : \si^{(\al)}_j \le A (S^{(\be)}_j)^{1/j}.
    \end{equation}
    Then $\fS$, $\fK$, $\ul \fL$, and $\fL$ are R-equivalent.
\end{theorem}

\begin{proof}
   By \Cref{thm:SKL}, it suffices to show $\fL  \{\preceq\} \fS$.
   Let $\al>0$ be fixed.
   By \eqref{eq:liminf}, there exist $\be>0$ and $C\ge 1$ such that
   \[
     \frac{k}{\mu^{(\be)}_k} \le C \sum_{j \ge k} \frac{1}{\mu^{(\al)}_j}.
   \]
   We may assume that $\al \le \be$ and hence $\mu^{(\al)} \le \mu^{(\be)}$.
   Thus we get
   \begin{equation*}
      \ta^{(\be)}_k = \frac{k}{\mu^{(\be)}_k}+\sum_{j\ge k}\frac{1}{\mu^{(\be)}_j}  \le (C+1) \sum_{\ell \ge k} \frac{1}{\mu^{(\al)}_\ell}
   \end{equation*}
   and consequently,
   \begin{equation*}
      \frac{k}{\si^{(\be)}_k} \le D \sum_{\ell \ge k} \frac{1}{\mu^{(\al)}_\ell}.
   \end{equation*}
   Then, by \eqref{eq:roquS}, there exist $\ga>0$ and $A\ge 1$ such that
   \[
      (L^{(\al)}_k)^{1/k} \le  \frac{k}{\sum_{\ell \ge k} (\mu^{(\al)}_\ell)^{-1}} \le D \si^{(\be)}_k \le AD  (S^{(\ga)}_k)^{1/k}
   \]
   and we are done.
\end{proof}

The conclusion of \Cref{thm:sufficientRequi} (invoking \Cref{thm:moment}) means that
\[
\La^{\{\fS\}} = \La^{\{\fK\}} = \La^{\{\fQ\}} = \La^{\{\ul \fL\}} = \La^{\{\fL\}}.
\]
So under the assumption of the theorem all presented solutions to the mixed Borel problem coincide with the optimal one; cf.\ \Cref{sec:L}.
We shall see in \Cref{thm:fKfLom} that for non-quasianalytic Braun--Meise--Taylor classes similarly
optimality is achieved by the different approaches leading to $\La^{\{\ka\}}$,  $\La^{\{\fK\}} = \La^{\{\fQ\}}$, and $\La^{\{\ul \fL\}}$.

\begin{remark}
  Let us discuss what R-equivalence of $\fS$, $\fK$, $\ul \fL$, and $\fL$ implies for $\fM$.
  In view of \Cref{thm:SKL}, this means that for each $\al>0$ there exists $\be >0$ such that
  $(L^{(\al)}_k)^{1/k} \lesssim (S^{(\be)}_k)^{1/k}$. Since $(S^{(\be)}_k)^{1/k} \le \si^{(\be)}_k = \ta^{(\be)}_1 \frac{k}{\ta^{(\be)}_k}$,
  we may infer
  \[
    \Big(\Big(\frac{k}{\sum_{\ell \ge k} (\mu^{(\al)}_\ell)^{-1}} \Big)^{k-j_k} M^{(\al)}_{j_k} \Big)^{1/k} \lesssim \frac{k}{\ta^{(\be)}_k},
  \]
  where $j_k$ is an integer $0\le j <k$ where the minimum in the definition of $L^{(\al)}_k$ is attained.
  A simple conversion of terms gives
  \[
    \Big(\frac{1}{k}\sum_{\ell \ge k} \frac{1}{\mu^{(\al)}_\ell} \Big)^{j_k/k} (M^{(\al)}_{j_k})^{1/k}
    \lesssim \frac{1}{\ta^{(\be)}_k} \sum_{\ell \ge k} \frac{1}{\mu^{(\al)}_\ell}
  \]
  and so, by the definition of $\ta^{(\be)}_k$,
  \[
    \Big(\frac{1}{k}\sum_{\ell \ge k} \frac{1}{\mu^{(\al)}_\ell} \Big)^{j_k/k} (M^{(\al)}_{j_k})^{1/k}
    \lesssim \frac{\mu^{(\be)}_k}{k} \sum_{\ell \ge k} \frac{1}{\mu^{(\al)}_\ell}.
  \]
  We see that \eqref{eq:liminf} would follow if the left-hand side is bounded away from zero. 
\end{remark}

In the following lemma we give conditions \emph{purely for $\fM$} which imply the assumptions of \Cref{thm:sufficientRequi}.
Notice however that \eqref{eq:invmg} seems to be quite restrictive.

\begin{lemma}
  Let $\fM = \{M^{(\al)}: \al>0\}$ be a non-quasianalytic weight matrix
  such that
  $\mu^{(\al)} \le \mu^{(\be)}$ if $\al \le \be$,
  \begin{equation} \label{eq:liminf2}
     \A \al>0 \E \be>0 : \liminf_{k \to \infty} \frac{\mu^{(\be)}_k}{k} \sum_{j \ge 2k} \frac{1}{\mu^{(\al)}_j}>0,
  \end{equation}
  and
  \begin{equation} \label{eq:invmg}
     \A \al>0 \E \be>0 \E A\ge 1 \A j \in \N : (\mu^{(\al)}_j)^2 \le A \mu^{(\be)}_{2j}.
  \end{equation}
  Then \eqref{eq:liminf} and \eqref{eq:roquS} hold.
\end{lemma}

\begin{proof}
   Obviously, \eqref{eq:liminf2} implies \eqref{eq:liminf}.
   Next we claim that \eqref{eq:liminf2} implies
   \begin{equation} \label{eq:first}
      \A \al>0 \E \be>0 \E B\ge 1 \A j \in \N : \ta^{(\be)}_j \le B \ta^{(\al)}_{2j}.
   \end{equation}
   Indeed,
   \begin{align*}
      \ta^{(\be)}_j &= \frac{j}{\mu^{(\be)}_j}+\sum_{k\ge 2j}\frac{1}{\mu^{(\be)}_k} + \sum_{j \le k< 2j}\frac{1}{\mu^{(\be)}_k}
      \le \frac{2j}{\mu^{(\be)}_j}+\sum_{k\ge 2j}\frac{1}{\mu^{(\be)}_k}
      \\
      &\le C \sum_{k\ge 2j}\frac{1}{\mu^{(\al)}_k} + \sum_{k\ge 2j}\frac{1}{\mu^{(\be)}_k}
      \le (C+1) \sum_{k\ge 2j}\frac{1}{\mu^{(\al)}_k} \le (C+1) \ta^{(\al)}_{2j}.
  \end{align*}
  Consequently, since $\ta^{(\al)}_j$ is decreasing,
  \begin{align*}
    B^{2j} (\ta^{(\al)}_{1})^j (\ta^{(\al)}_{j})^j
     &\ge B^{2j} \ta^{(\al)}_{1} \cdots \ta^{(\al)}_{2j}
     \\
     &\ge B^{2j} (\ta^{(\al)}_{2})^2 (\ta^{(\al)}_{4})^2  \cdots (\ta^{(\al)}_{2j})^2
     \ge  (\ta^{(\be)}_1 \cdots \ta^{(\be)}_j)^2.
  \end{align*}
   Now \eqref{eq:invmg} implies that there exist $\ga>0$ and $A\ge 1$ such that
   \begin{align*}
      (\ta^{(\be)}_j)^2 \ge \frac{j^2}{(\mu^{(\be)}_j)^2} + \sum_{k \ge j} \frac{1}{(\mu^{(\be)}_k)^2}
      &\ge \frac{1}{2A} \frac{2j}{\mu^{(\ga)}_{2j}} + \frac{1}{A} \sum_{k \ge j} \frac{1}{\mu^{(\ga)}_{2k}}
      \\
      &\ge \frac{1}{2A} \Big( \frac{2j}{\mu^{(\ga)}_{2j}} +  \sum_{k \ge 2j} \frac{1}{\mu^{(\ga)}_{k}}\Big)
      = \frac{1}{2A} \ta^{(\ga)}_{2j}.
   \end{align*}
   In view of \eqref{eq:first} there exist $\de>0$ and $D\ge 1$ such that $\ta^{(\ga)}_{2j} \ge\frac{1}{D} \ta^{(\de)}_{j}$.
   Thus
   \begin{align*}
      B^{2j} (\ta^{(\al)}_{1})^j (\ta^{(\al)}_{j})^j
      &\ge (\ta^{(\be)}_1 \cdots \ta^{(\be)}_j)^2
      \ge (2AD)^{-j} \ta^{(\de)}_{1} \cdots \ta^{(\de)}_{j}.
   \end{align*}
   It is easy to see that
   that $\si^{(\al)}_j \le A (S^{(\de)}_j)^{1/j}$ is equivalent to
   \begin{equation*}
      \ta^{(\de)}_1 \cdots \ta^{(\de)}_j \le \widetilde A^j (\ta^{(\al)}_j)^j
   \end{equation*}
   and hence the statement is proved.
\end{proof}

\section{Classical cases} \label{sec:classical}

\subsection{Braun--Meise--Taylor classes}

The goal of this section is to show that,
for any non-quasianalytic weight function $\om$,
\[
 \La^{\{\ka\}} = \La^{\{\fK\}} = \La^{\{\fQ\}} = \La^{\{\ul \fL\}},
\]
where $\ka=\ka_\om$ and the families $\fK$, $\fQ$, and $\fL$ are derived from $\fM = \fM_\om$.

\begin{proposition} \label{prop:kafK}
  Let $\om$ be a non-quasianalytic weight function, $\ka=\ka_\om$, and $\fK = \fK(\fM_\om)$.
  Then $\La^{\{\ka\}} = \La^{\{\fK\}}$.
\end{proposition}

\begin{proof}
  Now $\ka$ and $\ka_\al$ are equivalent for all $\al>0$, since $\om$ and $\om_{M^{(\al)}}$ are, by \cite[Lemma 5.7]{RainerSchindl12}.
  So there is $C\in \N_{\ge 1}$ such that $\ka_\al \le C \ka +C$.
  Let $\fH = \fH_{\ka} = \{H^{(\al)}: \al>0\}$ be the weight matrix associated with $\ka$ (cf.\ \Cref{omegaproperties}).
  Then, for $\al =1$, assuming that $\ka$ and $\ka_1$ are normalized,
  \begin{align*}
     H^{(1)}_j = \exp(\vh^*_{\ka}(j)) = \sup_{t \ge 0} \frac{t^j}{\exp(\ka(t))} \le e \sup_{t \ge 0} \frac{t^j}{\exp(\frac{1}{C}\ka_1(t))}
     = e (K^{(1)}_{Cj})^{1/C}.
  \end{align*}
  For $x \in \N_{\ge 1}$, we conclude $H^{(x)}_j = (H^{(1)}_{xj})^{1/x} \le e^{1/x} (K^{(1)}_{Cxj})^{1/(Cx)}$
  and thus $\fH \{\preceq\} \fK$,
  since $\fK$ has R-moderate growth.

  We also have $\ka \le C \ka_\al +C$ for some $C\in \N_{\ge 1}$ so that an analogous computation gives
  $K^{(\al)}_j \le e (H^{(1)}_{Cj})^{1/C}$ and hence
  $\fK \{\preceq\} \fH$,
  since $\fH$ has R-moderate growth.
\end{proof}

As a technical tool we will associate with a weight sequence $M$ and a positive integer $n$
the weight sequence $M^{[n]}$ defined by $M^{[n]}_j := M_{nj}^{1/n}$.
Note that
\[
  \mu^{[n]}_j := \frac{M^{[n]}_j}{M^{[n]}_{j-1}} = (\mu_{n(j-1)+1}\cdots \mu_{nj})^{1/n}
\]
satisfies $\mu_{2j} \le A \mu^{[4]}_j$ for some constant $A\ge 1$, which follows easily from the fact that $\mu$ is increasing.
We always have $M \le M^{[n]}$, and $M^{[n]} \preceq M$ holds provided that $M$ has moderate growth.
If $M$ is non-quasianalytic, so is $M^{[n]}$.

\begin{lemma}\label{SVshifted}
Let $M$ be a non-quasianalytic weight sequence and $M'$ a positive sequence. Then
$M' \prec_{SV} M$ implies
${M'}^{[n]} \prec_{SV} M^{[4n]}$ for all positive integers $n$.
\end{lemma}

\begin{proof}
For all $0\le i<j$,
\begin{align*}
  \Big(\frac{{M'}^{[n]}_j}{s^j M^{[4n]}_{i}}\Big)^{\frac1{j-i}}
  &=
  \Big(\frac{M'_{nj}}{s^{nj} M_{ni}^{[4]}}\Big)^{\frac1{n(j-i)}}
  \le
  \Big(\frac{M'_{nj}}{s^{nj} M_{ni}}\Big)^{\frac1{n(j-i)}}
  \le \sup_{0 \le i < nj} \Big(\frac{M'_{nj}}{s^{nj} M_{i}}\Big)^{\frac1{nj-i}}.
\end{align*}
Moreover,
\[
\mu^{[4n]}_k=(\mu^{[4]}_{n(k-1)+1}\cdots\mu^{[4]}_{nk})^{1/n}
\ge \mu^{[4]}_{n(k-1)} \ge A^{-1} \mu_{2n(k-1)}
\]
so that, for $j \ge 2$,
\[
\sum_{k\ge j}\frac{1}{\mu^{[4n]}_k} \le A\sum_{k\ge nj}\frac{1}{\mu_k}.
\]
The assertion follows.
\end{proof}

\begin{proposition} \label{prop:omSVred}
  Let $\fM$ be a non-quasianalytic weight matrix of R-moderate growth and
  $\fM'$ a one-parameter family of positive sequences such that
  $\liminf_{k \to \infty} (M'_k/k!)^{1/k} >0$ for all $M' \in \fM'$.
  Assume that
  \begin{equation} \label{eq:SVred}
    \E \widehat M' \in \fM' \A M' \in \fM' \E n \in \N_{\ge 1} : M' \preceq (\widehat M')^{[n]}.
  \end{equation}
  If there is $\widehat M \in \fM$ such that $\widehat M' \prec_{SV} \widehat M$, then
  $\La^{\{\fM'\}} \subseteq j^\infty \cE^{\{\fM\}}(\R)$.
\end{proposition}

\begin{proof}
    Let $a \in\Lambda^{\{\fM'\}}$. Then $a \in\Lambda^{\{M'\}}$ for some $M' \in \fM'$.
    By \eqref{eq:SVred}, we find $n\ge 1$
    such that
    $M' \preceq (\widehat M')^{[n]}$.
    Consequently, $a \in\Lambda^{\{(\widehat M')^{[n]}\}}$.
    By \Cref{SVshifted}, we have $(\widehat M')^{[n]} \prec_{SV} \widehat M^{[4n]}$.
    We may conclude that $a \in j^\infty \cE^{\{\widehat M^{[4n]}\}}(\R)$; cf.\ \Cref{thm:DCSV}.
    Since $\fM$ has R-moderate growth,
    $\cE^{\{\widehat M^{[4n]}\}}(\R) \subseteq \cE^{\{\fM\}}(\R)$.
\end{proof}

\begin{lemma} \label{lem:omSVred}
  If $\om$ is a pre-weight function, then $\fM = \fM_\om$ satisfies \eqref{eq:SVred}.
\end{lemma}

\begin{proof}
  We have $M^{(n)} = (M^{(1)})^{[n]}$ for all positive integers $n$; cf.\ \eqref{eq:asswm}.
\end{proof}

\begin{theorem} \label{thm:fKfLom}
  Let $\om$ be a non-quasianalytic weight function and $\fM = \fM_\om$.
  Then the derived families $\fK$ and $\ul \fL$ are R-equivalent.
\end{theorem}

\begin{proof}
  By \Cref{thm:SKL}, it remains to show $\ul \fL \{\preceq\} \fK$.
  Fix $\ul L \in \ul \fL$ and $M \in \fM$ such that $L = L(M)$.
  We have $\ul L \preceq M$ and thus there is $H \ge 1$ such that
  $\om_{M}(t)\le \omega_{\ul L}(H t)$ for all $t\ge 0$.
  Now $\om_{M}$, being equivalent to $\om$  by \cite[Lemma 5.7]{RainerSchindl12}, is a weight function so that, for some $A \ge 1$,
  \[
    \om_{M}(t) \le  A \omega_{M}(tH^{-1})+A\le A \omega_{\ul L}(t)+A.
  \]
  That means $\omega_{\ul L} \preceq \om$ and thus $\fB \{\preceq\} \fM$, where $\fB = \{B^{(\al)}: \al>0\} = \fM_{\om_{\ul L}}$ is the weight matrix
  associated with $\omega_{\ul L}$ (cf.\ \Cref{omegaproperties}; we do not assume that $\om_{\ul L}$ is normalized). In general,
  $\omega_{\ul L}$ might just be a pre-weight function (not a weight function)
  so that only the inclusion
  $\La^{\{\omega_{\ul L}\}} \subseteq \La^{\{\fB\}}$ is available.
  Observe that $B^{(1)}$ and $\ul L$ are closely related: by \eqref{eq:minorant} and \eqref{eq:asswm}, 
  \begin{align*}
    \ul L_k =  \sup_{t\ge 0} \frac{t^k}{e^{\om_{\ul L}(t)}},
    \quad
     B^{(1)}_k = e^{\vh^*_{\om_{\ul L}}(k)} = \sup_{t\ge 1} \frac{t^k}{e^{\om_{\ul L}(t)}},
  \end{align*}
  and consequently
  \begin{align*}
     B^{(1)}_k \le \ul L_k &=  \max\Big\{\sup_{0\le t\le 1} \frac{t^k}{e^{\om_{\ul L}(t)}}, \sup_{t\ge 1} \frac{t^k}{e^{\om_{\ul L}(t)}}\Big\}
     \\
     &\le \max\Big\{1, \sup_{t\ge 1} \frac{t^k}{e^{\om_{\ul L}(t)}}\Big\} = \max\{1, B^{(1)}_k\}.
  \end{align*}
  Since $B^{(1)}_k \to \infty$, we see that the sequences $B^{(1)}$ and $\ul L$ are equivalent.

  By \Cref{lem:omSVred}, we have $B^{(n)}= (B^{(1)})^{[n]}$.
  Since $L \prec_{SV} M$ and thus $\ul L \prec_{SV} M$ and the sequences $B^{(1)}$ and $\ul L$ are equivalent,
  we find $\La^{\{\fB\}} \subseteq  j^\infty \cE^{\{\fM\}}(\R)$, by \Cref{prop:omSVred}.
  (We have $\omega_{\ul L}(t)= o(t)$ as $t \to \infty$ and consequently $(B^{(n)}_j/j!)^{1/j} \to \infty$ for all $n$, by \Cref{assofuncproper}.)
  Thus $\La^{\{\omega_{\ul L}\}} \subseteq j^\infty \cE^{\{\om\}}(\R)$
  which implies $\omega_{\ul L} \preceq \ka$, by \Cref{lem:A1}.

  It follows that
  $B^{(1)} \in \La^{B^{(1)}}_1 = \La^{\om_{\ul L}}_1 \subseteq  \La^{\{\om_{\ul L}\}} \subseteq \La^{\{\ka\}} = \La^{\{\fK\}}$, by \Cref{prop:kafK},
  and consequently there exists $K \in \fK$ such that $\ul L \preceq K$, since $B^{(1)}$ and $\ul L$ are equivalent.
\end{proof}

Now \eqref{eq:main2} follows from \Cref{prop:KW}, \Cref{prop:kafK}, and \Cref{thm:fKfLom}.

Note that a necessary and sufficient condition for $\om$ being a strong weight function
is that $\fK$, $\fQ$, $\ul \fL$, and $\fL$ are all R-equivalent to $\fM_\om$.

\subsection{Denjoy--Carleman classes}

\begin{theorem}
  Let $M$ be a non-quasianalytic weight sequence of moderate growth.
  Then the derived sequences $S$, $K$, and $L$ are equivalent.
\end{theorem}

\begin{proof}
   We have $S \preceq K \preceq L$ by \Cref{thm:SKL}.
   That $L \preceq S$ was shown in \cite[Theorem 3.10]{maximal}.
\end{proof}

\begin{remark}
  The assumption that $M$ has moderate growth can be replaced by the weaker condition
  \begin{equation*}
     \liminf_{k \to \infty} \frac{\mu_k}{k} \sum_{j \ge 2k} \frac{1}{\mu_j}>0,
  \end{equation*}
  which guarantees that $S$ has moderate growth and $L \preceq S$;
  cf.\ \cite[Theorem 3.10]{maximal}.
\end{remark}

\appendix

\section{} \label{appendix}

The goal of this section is to prove the following proposition.
It is due to \cite{BonetMeiseTaylor92} if $\si$ is a weight function.
We will show that it is valid if $\si$ is just a pre-weight function
by slightly modifying the proof of \cite{BonetMeiseTaylor92}.
The proposition is used in this more general form in the proof of \Cref{thm:fKfLom}.

\begin{proposition} \label{lem:A1}
    Let $\om$ be a non-quasianalytic weight function.
    Let $\si$ be a pre-weight function such that $\si(t)=o(t)$ as $t \to \infty$.
    Then $\La^{\{\si\}} \subseteq j^\infty \cE^{\{\om\}}(\R)$ implies $\si \preceq \ka_\om$.
\end{proposition}

For a pre-weight function $\si$ we have $\Lambda^{\{\si\}} \subseteq \Lambda^{\{\fM_\si\}}$ (and in general not equality).
But since $\om$ is a weight function, $\La^{\{\si\}} \subseteq j^\infty \cE^{\{\om\}}(\R)$ entails
\begin{equation} \label{eq:A1}
    \Lambda^{\{\fM_\si\}} \subseteq j^\infty \cE^{\{\om\}}(\R).
\end{equation}
Indeed, let $a \in \Lambda^{\{\fM_\si\}}$, i.e., $|a_k| \le C H^{k} e^{\frac{1}{\al}\vh_\si^*(\al k)}$ for some $C,H\ge1$ and $\al>0$.
Then $b = (H^{-k} a_k)_k$ belongs to $\Lambda^{\{\si\}}$.
If $g \in \cE^{\{\om\}}(\R)$ satisfies $j^\infty g=b$, then
$f(x):= g(Hx)$ fulfills $j^\infty f=a$ and belongs to $\cE^{\{\om\}}(\R)$ since $\om$ is a weight function (cf.\ \eqref{newexpabsorb}).

Our goal is now to show that \eqref{eq:A1} implies $\si \preceq \ka_\om$.
The one crucial task is to identify the dual of $\Lambda^{\{\fM_\si\}}$.
Let us write $\fS = \{S^{(\al)} : \al>0\} := \fM_\si$ throughout this section.

\begin{lemma}
  \label{lem:dfs}
  Let $\Lambda^k := \La^{S^{(k)}}_k = \{a \in \C^\N:  \sup_{j\in \N} \frac{|a_j|}{k^jS^{(k)}_j}< \infty \}$ for $k \in \N_{\ge 1}$. Then
  the inclusion
  \[
  \Lambda^k \hookrightarrow \Lambda^{k+1}
  \]
  is compact. In particular $\Lambda^{\{\fM_\si\}} (=\on{ind}_{k\in \N} \Lambda^k )$ is a (DFS)-space (cf. \cite[25.20]{MeiseVogt97}).
\end{lemma}

\begin{proof}
Let $(a^{(n)})_n$ be a bounded sequence in $\La^{k}$. Then $\big\{b^{(n)}_j:= \frac{a^{(n)}_j}{k^j S^{(k)}_j} : n,j \in \N\big\}$ is bounded.
Thus the sequence $(b^{(n)}_1)$ is bounded and after passing to a subsequence we may assume that it is convergent.
Passing to a subsequence again we may assume that $(b^{(n)}_2)$ converges.
Iterating this procedure and taking the diagonal sequence,
we end up with a subsequence $b^{(n_j)}$ such that $b^{(n_j)}_i \to c_i \in \C$ for all $i \in \N$.
To finish the proof we show that
$a^{(n_j)} \rightarrow c$ in $\La^{k+1}$, where $c =(c_i)$.
It is clear that for fixed $\ve>0$ there exists $i_0$ such that for all $i \ge i_0$ and all $j$, we have
\[
\frac{|a_i^{(n_j)}|}{(k+1)^iS^{(k+1)}_i} = |b_i^{(n_j)}| \frac{k^i S^{(k)}_i}{(k+1)^iS^{(k+1)}_i}  \le \ve
\quad \text{ and }\quad \frac{|c_i|}{(k+1)^iS^{(k+1)}_i}  \le \ve.
\]
In addition there exists $j_0$ such that  $|a^{(n_j)}_i - c_i| \le \ve$ for $j \ge j_0$ and $i \le i_0$.
This yields that $a^{(n_j)} \rightarrow c$ in $\La^{k+1}$.
\end{proof}

\begin{lemma} \label{lem:A3}
  Let $\si$ be a pre-weight function.
  Then
  \[
  (\La^{\{\fM_\si\}})' \cong \{f \in \cH(\C): \A n \in \N \E A >0 \A z \in \C:~ |f(z)| \le A e^{\om_{S^{(n)}} (\frac{|z|}n)}  \} =: \cA^0
  \]
  The isomorphism is explicitly given by
  \[
  T \mapsto \Ph(T):= \Big(z \mapsto \sum_{j}T(e_j)z^j\Big),
  \]
  where $e_j$ denotes the $j$-th unit vector.
\end{lemma}

\begin{proof}
  As a consequence of the compactness of the connecting mappings (\Cref{lem:dfs})
  the collection $\{B_n\}_n$ of closed unit balls $B_n \subseteq \La^n$
  forms a fundamental system of bounded sets in $\La^{\{\fM_\si\}}$(cf.\  \cite[25.19]{MeiseVogt97}).
  So a set $B$ is bounded in $\La^{\{\fM_\si\}}$ if and only if there exist $n \in \N$ and $\la > 0$ such that $B \subseteq \la B_n$.
  Therefore a 0-neighborhood base in the (strong) dual is given by the collection of the polars $(nB_n)^\circ$.
  Let $T \in (2nB_{2n})^\circ$. Then,
  \[
  |T(e_j)| \le \frac{1}{2n} \frac{1}{(2n)^j S^{(2n)}_j}
  \]
  and hence
  \begin{align*}
     |\Ph(T)(z)| \le \sum_{j}|T(e_j)| |z|^j &\le \frac{1}{2n} \sum_{j} \frac{|z|^j}{(2n)^j S^{(2n)}_j}
     \\
     &\le \frac{1}{2n} \sup_{k \in \N} \frac{|z|^k}{n^k S^{(n)}_k} \sum_{j\ge 0} 2^{-j} = \frac{1}{n} e^{\om_{S^{(n)}}(\frac{|z|}n)}.
  \end{align*}
  Thus $\Ph((2nB_{2n})^\circ)$ is contained in the ball of radius $\frac{1}{n}$ with respect to the weight $e^{\om_{S^{(n)}}(\frac{|z|}n)}$, whence $\Ph$ is continuous.

  Conversely, let $f(z) = \sum_j c_j z^j \in \cA^0$ be such that
  \[
  |f(z)| \le \frac{1}{2n} e^{\om_{S^{(2n)}} (\frac{|z|}{2n})}.
  \]
  Then, by the Cauchy estimates,
  \begin{align*}
      |c_j| \le \frac{1}{2n} \inf_{r>0} \frac{e^{\om_{S^{(2n)}} (\frac{r}{2n})} }{r^j}= \frac{1}{2n} \Big(\sup_{r>0} \frac{r^j}{e^{\om_{S^{(2n)}} (\frac{r}{2n})} } \Big)^{-1}
      = \frac{1}{(2n)^{j+1} S^{(2n)}_j}.
  \end{align*}
  Therefore, if $a \in nB_n$ then
  \[
  |\Ph^{-1}(f)(a)| \le \sum_{j\ge 0}\frac{n^{j+1} S^{(n)}_j}{(2n)^{j+1}S^{(2n)}_j}  \le 1,
  \]
  and so $\Ph^{-1}(f) \in (nB_n)^\circ$, which shows continuity of $\Ph^{-1}$.
\end{proof}

Now we are ready to show that \eqref{eq:A1} implies $\si \preceq \ka_\om$. If $\si \not\preceq \ka_\om$, then in view of \eqref{lem:BMT}
we can find a sequence of positive real numbers $a_j \to \infty$ such that
\begin{gather}
\label{eq:1}
  6j\si(a_j) \le P_\om(i a_j),\\
\label{eq:2}
  \log(1+|z|^2) \le \frac{1}{j} \si(|z|) \quad \text{for } |z|\ge a_j.
\end{gather}
By the proof of Proposition 2.4 in \cite{BonetMeiseTaylor89}, there exists of a sequence of polynomials $h_j$ and constants $C$ and $m$ such that
\[
|h_j(z)| \le C(1+a_j^2)^2e^{\frac{m}{j} P_\om (z)}, \quad z \in \C,
\quad\text{ and }\quad
h_j(i a_j) = e^{\frac{1}{j}P_\om(i a_j)}.
\]
Now set
\[
f_j(z):= e^{(\si(a_j) - \frac{1}{j} P_\om(i a_j))}h_j(z).
\]
Using \eqref{eq:2} and \eqref{eq:1},  we get
\[
|f_j(z)| \le C e^{(1+\frac{2}{j})\si(a_j) - \frac{1}{j} P_\om(ia_j) } e^{\frac{m}{j}P_\om(z)} \le C e^{-\frac{1}{2j}P_\om(ia_j)}e^{\frac{m}{j}P_\om(z)}.
\]
Applying the estimate
$P_\om(x+iy)\le |y|+A(\om(x)+1)$ (see \cite[Lemma 2.2]{BMT90}),
where $A$ is some absolute constant, gives
\[
|f_j(z)| \le C e^{Am} e^{-\frac{1}{2j}P_\om(i a_j)} e^{m|y| + \frac{Am}{j}\om(|z|)}.
\]
It follows that a subsequence of
$(f_j)$ tends to $0$ in $(\cE^{\{\om\}}(\R))'$ by means of the identification in \cite[Remark 1.4]{BonetMeiseTaylor92}.
On the other hand,
\[
f_j(ia_j)=e^{\si(a_j)} \ge e^{\omega_{S^{(n)}}(\frac{a_j}n)},
\]
since $\sigma(t) \ge \omega_{S^{(1)}}(t) \ge  \omega_{S^{(n)}}(t) \ge \omega_{S^{(n)}}(t/n)$.
Thus, no subsequence of $(f_j)$ converges to $0$ in $(\La^{\{\fM_\si\}})' \cong \cA^0$; see \Cref{lem:A3}.
Now \cite[Corollary 2.2]{BonetMeiseTaylor92} gives that $\Lambda^{\{\fM_\si\}}$ is not contained in $j^\infty \cE^{\{\om\}}(\R)$ (for this we need the (DFS)-property, i.e., Lemma \ref{lem:dfs}).
So the assumption $\si \not\preceq \ka_\om$ contradicts \eqref{eq:A1} and the proof of \Cref{lem:A1} is complete.

\begin{remark}
  The Beurling version of \Cref{lem:A1} is valid, too.
  We omit the proof.
\end{remark}


\def\cprime{$'$}
\providecommand{\bysame}{\leavevmode\hbox to3em{\hrulefill}\thinspace}
\providecommand{\MR}{\relax\ifhmode\unskip\space\fi MR }
\providecommand{\MRhref}[2]{%
  \href{http://www.ams.org/mathscinet-getitem?mr=#1}{#2}
}
\providecommand{\href}[2]{#2}

\end{document}